\documentclass[12pt]{amsart}

\newtheorem{theorem}{Theorem}[section]

\newtheorem{lemma}[theorem]{Lemma}
\newtheorem{proposition}[theorem]{Proposition}

\theoremstyle{definition}

\theoremstyle{remark}

\numberwithin{equation}{section}

\newcommand{\be}{\begin{equation}}

\newcommand{\ee}{\end{equation}}
\newcommand{\bn}{\begin{eqnarray}}
\newcommand{\en}{\end{eqnarray}}
\newcommand{\bd}{\begin{displaymath}}
\newcommand{\ed}{\end{displaymath}}
\newcommand{\bnn}{\begin{eqnarray*}}
\newcommand{\enn}{\end{eqnarray*}}

\begin{document}

\title[A crossed product of the CAR algebra] {A crossed product of the CAR algebra in the Cuntz algebra}

%
\author{M.A.\,Aukhadiev, A.S.\,Nikitin, A.S.\,Sitdikov}
\address[M.A.\,Aukhadiev, A.S.\,Nikitin, A.S.\,Sitdikov]{Kazan State Power Engineering University\\
        Kazan, Russia, 420066}
\email[M.A.\,Aukhadiev]{m.aukhadiev@gmail.com}
\email[A.S.\,Nikitin]{drnikitin@rambler.ru }
\email[A.S.\,Sitdikov]{airat\_vm@rambler.ru}

\thanks{Research supported in part by Russian Foundation for Basic
Research, Grants 14-01-31358, 13-02-97054, 12-01-97016.}

\date{April 2, 2014}

\begin{abstract}
In this paper we show that the Cuntz algebra can be represented as a C*-crossed product by endomorphism of the canonical anticommutation relations (CAR) algebra, generated by the standard recursive fermion system.
\end{abstract}
\maketitle

\section{Introduction}
The canonical anticommutation relations algebra (CAR), generated by fermion creation and annihilation operators, with its representations, has been  well established in both physical and mathematical aspects by the moment \cite{Kawamura}. The mathematical research of the CAR algebra is based primarily on the operator algebras theory \cite{Bratelli}.

 The recursive construction of the CAR algebra based on the Cuntz algebra $\mathcal{O}_2$ generators was described in \cite{Kawamura}. 
Since the Cuntz algebra is finitely generated, this method is an effective instrument for research on fermion systems properties in the framework and terms of the algebra $\mathcal{O}_2.$ 
Moreover, such construction is extremely useful for investigation of the quantum systems superselection structure \cite{DR} in the  algebraic quantum field theory framework \cite{Horuzhi}.
The superselection sectors are the unitary equivalence classes of irreducible representations of the observable algebra, which satisfy the so-called Doplicher-Haag-Roberts and Bucholz-Fredenhagen selection criterions \cite{DH,Buchholz}. Such representations can be described by means of localized endomorphisms of this algebra. Hence, each sector is identified with a set of unitary equivalent localized endomorphisms. These endomorphisms form a symmetric $C^*$-tensor category with  intertwining operators between localized endomorphisms as category morphisms between sectors. 
Hence, one can extend the observable algebra by the crossed product with the endomorphisms category\cite{DR2}. By the Doplicher-Roberts duality theorem \cite{DR3}, the automorphisms group of the resulting $C^*$-algebra (the field algebra) is a compact group
and the observable algebra is a subalgebra of the field algebra, consisting of fixed points with respect to the action of this group.

 In this work we consider the mathematical part of the problem. Using an injection of the CAR algebra into the Cuntz algebra, we construct a crossed product of this subalgebra by the group of integers with respect to one endomorphism of this algebra. We show that the resulting $C^\ast$-algebra coincides with the Cuntz algebra. In other words, the Cuntz algebra can be described as a $C^\ast$-crossed product of the CAR algebra, generated by the recursive fermion system \cite{Kawamura}.
 
 The first part of this work contains preliminaries. The second part is a desciption of $C^\ast$-crossed product construction of the CAR algebra in the Cuntz algebra in the framework shown in \cite{Anton}. Further we show that the constructed crossed product is isomorphic to the Cuntz algebra. We compare this construction to the well-known crossed product of the Cuntz-Krieger algebra.

The research is partially supposrted by RFBR grants 14-01-31358, 13-02-97054, 12-01-97016.

\section{Preliminaries}

The Cuntz algebra \cite{Cuntz, DR4} $\mathcal{O}_d$ $(d\geq 2)$ is a $C^*$-algebra, generated by isometries $\psi_1,\ \psi_2,\
\ldots,\ \psi_d$, which satisfy the following conditions:
\bn
&& \psi_i^* \, \psi_j = \delta_{i,j} I,                     \label{CR1}\\
&& \sum_{i=1}^d \psi_i \, \psi_i^* = I, \label{CR2}
\en
where $I$ is a unit in the algebra. The following standard notation is introduced for convenience:
$\psi_{i_1i_2\cdots i_m}\equiv \psi_{i_1}\psi_{i_2}\cdots
\psi_{i_m}$, $\psi^*_{i_1i_2\cdots i_m}\equiv \psi^*_{i_m}\dots
\psi^*_{i_2} \psi^*_{i_1}$ and $\psi_{i_1\cdots i_m;\,j_n\cdots
j_1}\equiv \psi_{i_1}\cdots \psi_{i_m} \psi_{j_n}^* \cdots
\psi_{j_1}^*$. Conditions \ref{CR1}, \ref{CR2} imply that the algebra $\mathcal{O}_d$ is generated by so-called \emph{monomials} -- operators of type $\psi_{i_1\cdots i_m;\,j_n\cdots j_1}$ as a linear space.

Now we define the \emph{canonical} unital $\ast$-endomorphism $\rho$ on the algebra $\mathcal{O}_d$.
\be
\rho(X) = \sum_{i=1}^d \psi_i X \psi_i^*, \quad X \in
\mathcal{O}_d. \label{c-endo}
\ee

It is well-known that the generators $a_m$ and $a_n^*$ $(m,\,n=1,\,2,\,\ldots)$
of the $C^*$-algebra CAR of fermions satisfy the following relations 
\begin{eqnarray}
&& \{ a_m, \, a_n \} = \{ a^*_m, \, a^*_n \} = 0,     \label{car-1}\\
&& \{ a_m, \, a^*_n \} = \delta_{m,n}I. \label{car-2}
\end{eqnarray}

In paper \cite{Kawamura} K. Kawamura showed that the CAr algebra is isomorphic to $\mathcal{O}^{U(1)}_2\subset\mathcal{O}_2$, which consists of such elements in $\mathcal{O}_2$, which are invariant under the standard action of the group $U(1)$. In other words, this subalgebra is generated by monomials
\begin{eqnarray}
&& \psi_{i_1} \cdots \psi_{i_k} \psi^*_{j_k} \cdots \psi^*_{j_1},
\label{UHF2}
\end{eqnarray}
where $i_1,\ \ldots,\ i_k,\ j_1,\ \ldots,\ j_k = 1,\ 2$. The action 
$\tau$ of group $U(1)$ on $\mathcal{O}_2$ is given by
\begin{equation} \tau(z)(\psi_i)= z \psi_i, \quad z\in U(1)\
\ i=1,\,2.
\end{equation}

Following the work \cite{Kawamura} we embed the CAR algebra into 
$\mathcal{O}_2$ by means of recursive construction, which is called  {\it the recursive fermion system (RFS)\/}. We give its definition below. In the above-mentioned work it is shown that there exists a map $\zeta$ on
$\mathcal{O}_{2d}$ such that
\begin{equation}
a_n = \zeta^{n-1}(\bold a), \quad n=1,\,2,\,\ldots
\end{equation}
satisfying \eqref{car-1} and \eqref{car-2} for a fixed element $\bold
a\in\mathcal{O}_{2d}$.

Let $\bold a\in\mathcal{O}_{d}$, $\zeta:\mathcal{O}_{d}\to\mathcal{O}_{d}$ be a linear map and $\varphi$ be a unital $\ast$-endomorphism
on $\mathcal{O}_{d}$. A triple $R=(\bold a, \,
\zeta, \, \varphi)$ is called {\it a recursive fermion system \/} in $\mathcal{O}_{d}$, if it satisfies the following conditions:

\begin{equation} \label{RFS1}
\bold a^2 = 0,\ \{ \bold a, \, \bold a^* \} = I,
\end{equation}
\begin{equation}        \label{RFS2}               
 \{ \bold a, \, \zeta(X) \} = 0, \ \ \zeta(X)^*=\zeta(X^*),
    X \in \mathcal{O}_{2},
\end{equation}           
\begin{equation} \label{RFS3}
\zeta(X)\zeta(Y)=\varphi(XY), \quad
    X,\,Y \in \mathcal{O}_{2}.                            
\end{equation}

 The embedding $\varPhi_R$ of the CAR algebra into $\mathcal{O}_{d}$, 
corresponding to $R=(\bold a, \, \zeta, \, \varphi)$ is defined by the image of generators $a_n$ $(n=1,\,2,\,\ldots)$ of the CAR algebra in the following way.
\begin{equation}
\varPhi_R(a_n) \equiv \zeta^{n-1}(\bold a) \equiv
(\,\underbrace{\zeta\circ\zeta\circ\cdots\circ\zeta}_{n-1}\,)(\bold
a), \quad n=1,\,2,\,\ldots. \label{RFS4}
\end{equation}
By virtue of conditions \eqref{RFS1}--\eqref{RFS3}, the following equations are verified.
$$ \{ \varPhi_R(a_m), \, \varPhi_R(a_n) \} = \varphi^{m-1}(\{
\bold a, \, \zeta^{n-m}(\bold a) \})
= \varphi^{m-1}(0) = 0, \quad m\leq n, $$
$$ \{ \varPhi_R(a_m), \, \varPhi_R(a_n)^* \} = \varphi^{m-1}(\{
\bold a, \, \varphi^{n-m}(\bold a^*)\})
= \varphi^{m-1}(0)=0, \quad m < n, $$
$$ \{ \varPhi_R(a_n), \, \varPhi_R(a_n)^* \} = \varphi^{n-1}(\{
\bold a, \, \bold a^*\}) =\varphi^{n-1}(I)=I.
$$
Denote by $\mathcal{A}_R \subset\mathcal{O}_{d}$ the image of embedding. $\mathcal{A}_R$ is called \emph{a CAR-subalgebra, corresponding to} $R$.

Consider the case $d=2$, and define the standard recursive fermion system $C=(\bold
a,\,\zeta,\,\varphi)$.
\begin{eqnarray}
&& \bold a\equiv \psi_1 \psi_2^*,                            \label{RFS5}\\
&& \zeta(X)\equiv \psi_1 X \psi_1^* - \psi_2 X \psi_2^*, \quad X
\in \mathcal{O}_2,
                                                      \label{RFS6}\\
&& \varphi(X)\equiv\rho(X)= \psi_1 X \psi_1^* + \psi_2 X \psi_2^*,
\quad
       X \in \mathcal{O}_2.                                 \label{RFS7}
\end{eqnarray}
Here  $\rho$ is the canonical endomorphism of the algebra  $\mathcal{O}_2$ \eqref{c-endo}. The CAR-subalgebra, corresponding to the standard recursive fermion system $C$ is denoted by
$\mathcal{A}_C$. We have  $\mathcal{A}_C=\mathcal{O}_2^{U(1)}$ (see \cite{Kawamura} for details).

\section{Crossed product}

In this part we construct the $C^\ast$-crossed product of the algebra 
$\mathcal{A}_C$ by an endomorphism $\delta$ of this algebra in the framework of $C^\ast$-crossed product, described in \cite{Anton,Lebedev}. Consider a map $\delta\colon \mathcal{A}_C\to \mathcal{O}_2$, for any $a\in  \mathcal{A}_C$ given by
\begin{equation}\label{delta}
\delta(a)=\psi_1a\psi_1^*.
\end{equation}

\begin{lemma}\label{hom}
The map $\delta$ is a ${}^\ast$-endomorphism $\delta\colon \mathcal{A}_C\to\mathcal{A}_C$.
\end{lemma}
\begin{proof}
 The algebra $\mathcal{A}_C$ is a $C^\ast$-subalgebra in $\mathcal{O}_2$, generated by monomials of type \eqref{UHF2}. The formula \eqref{delta} implies that $\delta$ is linear and involution-preserving. Since $\psi_1$ is an isometry, $\delta$ is a homomorphism.  
 
Let $a\in\mathcal{O}_2$ be a monomial of type \eqref{UHF2} for $k=j$. Then $\delta(a)$ is a monomial of type \eqref{UHF2} for $k=j+1$. Therefore, $\delta(a)\in \mathcal{A}_C$. It is easy to show that $\delta$ is continuous. Thus, the image of $\mathcal{A}_C$ under $\delta$ is contained in $\mathcal{A}_C$ and $\delta$ is a ${}^\ast$-endomorphism on $\mathcal{A}_C$. \end{proof}

Let $\gamma$ be a  $^\ast$-endomorphism on a $C^\ast$-algebra $\mathcal{A}$. A linear positive continuous map $\gamma_*\colon\mathcal{A}\to\mathcal{A}$, preserving involution is called \emph{a transfer operator} (with respect to $\gamma$), if for any $a,b\in\mathcal{A}$ the following condition is verified  \cite{Lebedev}.
 \begin{equation}
 \label{tr}
\gamma_*(\gamma(a)b)=a\gamma_*(b) 
 \end{equation}
If moreover  one has  \begin{equation}
 \label{tr2}\gamma\gamma_*(a)=\gamma(1)a\gamma(1),\end{equation}
for any $a\in\mathcal{A}$, then the transfer operator $\gamma_*$ is called \emph{full}.

 Define the following map on the algebra $\mathcal{A}_C$ 
\begin{equation}
\delta_*(a)=\psi_1^*a\psi_1 \mbox{ for any } a\in \mathcal{A}_C \label{transfer}
\end{equation}

\begin{lemma} The map
 $\delta_*$, given by \eqref{transfer} is a full transfer operator  $\delta_*\colon\mathcal{A}_C\rightarrow\mathcal{A}_C$ with respect to $\delta$.
 \end{lemma}
 \begin{proof}
 Similarly to the proof of Lemma \ref{hom} one can show that $\delta_*$ is linear continuous involution-preserving, and $\delta_*(a)\in\mathcal{A}_C$ for any $a\in\mathcal{A}_C$. Positiveness of this map is obvious.
 
  For any $a,b\in\mathcal{A}_C$ we have
   $$\delta_*(\delta(a)b)=\psi_1^*\psi_1a\psi_1^*b\psi_1=a\delta_*(b).$$
   This fact implies equation \eqref{tr}. 
   
It remains to show that this transfer operator is full. Indeed, for any $a\in\mathcal{A}_C$ we have
$$\delta\delta_*(a)=\psi_1\psi_1^*a\psi_1\psi_1^*=\delta(1)a\delta(1),$$
since $\delta(1)=\psi_1I\psi_1^*=\psi_1\psi_1^*$. This implies \eqref{tr2}.
   
   \end{proof}

\begin{proposition} The following conditions are satisfied in the $C^\ast$-subalgebra of the Cuntz algebra, corresponding to the Recursive Fermion System, and the transfer operator $\delta_*$.
 
\begin{equation}
\delta_*\zeta(X)=X,\hspace{0.3cm}X\in\mathcal{A}_C \label{transfer2}
\end{equation}

\begin{equation}\label{transfer3}
\delta_*(\mathbf{a})=0,
\hspace{0.3cm}n=0.\end{equation}
\end{proposition}

\begin{proof}
Equation \eqref{transfer3} holds, since $\mathbf{a}=\psi_1\psi_2^*$:
 $$\delta_*(\mathbf{a})=\psi_1^*\psi_1\psi_2^*\psi_1=0.$$

By definition of $\zeta$ \eqref{RFS6}, we have
$$\delta_*\zeta(X)=\psi_1^*\zeta(X)\psi_1=$$ 
$$=\psi_1^*(\psi_1 X \psi_1^*-\psi_2 X \psi_2^*)\psi_1=$$ $$=\psi_1^*\psi_1 X \psi_1^*\psi_1-\psi_1^*\psi_2 X \psi_2^*\psi_1=X.$$
\end{proof}

Consider a $C^*$-algebra
$\mathcal{B}:=\mathcal{B}(\mathcal{A}_C,\psi_1)$, generated by the algebra
 $\mathcal{A}_C$ and the isometry $\psi_1$. By definition, given in  \cite{Lebedev}, $\mathcal{A}_C$ is \emph{a coefficient algebra} for $\mathcal{B}$ if additionally the following is satisfied.

 \begin{equation}
\label{k1} \psi_1a=\delta(a)\psi_1,\ a\in \mathcal{A}_C.
\end{equation}

In our case this condition is satisfied, and we have the following.
\begin{lemma} Algebra $\mathcal{A}_C$ is a coefficient algebra for $\mathcal{B}$.
\end{lemma}

Denote by $\mathcal{B}_0$ a vector space, consisting of finite sums:
\begin{equation}
\label{x}x=\psi_1^{\ast N}a_{\overline{N}}+...+\psi_1^\ast a_{\overline{1}}+ a_0+a_1\psi_1+...+a_N\psi_1^N,
\end{equation}  

where $a_k,a_{\overline{l}}\in\mathcal{A}_C$, $N\in\mathbb{N}\cup \{0\}$. Due to results of \cite{Anton,Lebedev}, $\mathcal{B}_0$ is a dense *-subalgebra in the $C^*$-algebra $\mathcal{B}$. Denote $\mathcal{A}_k=a_k\psi_1^k$ and $\mathcal{A}_{-k}=\psi_1^{\ast k}a_{\overline{k}}$.

The following condition, denoted by (*),  provides uniqueness of decomposition \eqref{x} and coefficients $a_k,a_{\overline{l}}$  \cite{Lebedev}:
 \begin{equation}
 ||a_0||\leq||x||
 \end{equation}
 for any $x\in\mathcal{B}_0$ of type \eqref{x}.
 
 Let $C(S^1,\mathcal{O}_2)$ denote the $C^\ast$-algebra of all continuous functions on the unit circle $S^1$ taking values in the algebra $\mathcal{O}_2$, with uniform norm:
 $$||a||=\sup\limits_{z\in S^1}||a(z)||, \mbox{ where } a\in C(S^1,\mathcal{O}_2).  $$
 Every $\mathcal{O}_2$-valued function $a\in C(S^1,\mathcal{O}_2)$ can be represented in a form of Fourier series
 \begin{equation}\label{fur}
 a(z)\approx\sum\limits_{n=-\infty}^{+\infty}z^n a_n,\ a_n=\frac{1}{2\pi}\int\limits_{S^1}a(e^{i\theta})e^{-in\theta}d\theta\in\mathcal{O}_2.
 \end{equation}
 
 And each element $a\in C(S^1,\mathcal{O}_2)$ can be approximated  in norm of the algebra $ C(S^1,\mathcal{O}_2)$ by finite linear combinations of the form (\ref{fur}).
 
 Using the action $\tau$ of the circle $S^1$ on $\mathcal{O}_2$ for any monomial $V\in\mathcal{O}_2$ define a function $\widetilde{V}\in C(S^1,\mathcal{O}_2)$ by the formula:
 $$\widetilde{V}(z)=\tau(z)(V).$$
 Denote by $\widetilde{\mathcal{O}_2}$ a closed subalgebra in the algebra  $C(S^1,\mathcal{O}_2)$, generated by functions of type $\widetilde{V}$. 
 
 One can show that Fourier coefficients in (\ref{fur}) $a_k$ lie in  corresponding spaces $\mathcal{A}_k$, using the same proof as presented in \cite{AGL}. Algebras $\widetilde{\mathcal{O}_2}$ and $\mathcal{O}_2$ are isomorphic. Therefore, for any $a\in \mathcal{O}_2$ we have
 $$||a_0||=|| \frac{1}{2\pi}\int\limits_{S^1}\widetilde{a}(e^{i\theta})e^{-in\theta}d\theta ||\leq||a ||.$$
Thus, condition (*) holds.

  According to \cite{Anton}, the algebra $\mathcal{B}$ is considered as  the crossed product 
$$\mathcal{B}=\mathcal{A}_C\times_{\delta}\mathbb{Z}=
\mathcal{B}(\mathcal{A}_C,\psi_1).$$ 
The elements of this crossed product are finite sums of the form \eqref{x}.

\begin{proposition} The $C^*$-algebra
$\mathcal{A}_C\times_{\delta}\mathbb{Z}$ is the Cuntz algebra.
\end{proposition}
\begin{proof} The generators of the Cuntz algebra $\psi_1$ and $\psi_2$ have a form \eqref{x}. Indeed, taking $a_1=\psi_1\psi_1^*$ and  $a_i=0$ for $i=0,2,3,...$ we get $x=a_1\psi_1=\psi_1\psi^*_1\psi_1=\psi_1$. And if
$a_1=\psi_2\psi_1^*$ and $a_i=0,$ $i=0,2,3,...$, then
$x=a_1\psi_1=\psi_2\psi_1^*\psi_1=\psi_2$.
\end{proof}

An example of crossed product, considered in \cite{Anton}, is the Cuntz-Krieger algebra $\mathcal{O}_A$ -- a  $C^*$-algebra, generated by partial isometries $Q_i=S_i^*S_i$ $P_i=S_iS_i^*$, satisfying conditions 
$$P_iP_j=0, i\neq j; \hspace{3cm} Q_i=\sum_{r=1}^rA(i,r)P_r,$$ where $A$ is an $n\times n$-matrix with $A(i,j)\in\{0,1\}$, where each row and column is non-zero.

Denote $S_{\O}=1,$ $S_{\mu}=S_{i_{1}}S_{i_{2}}\cdot\cdot\cdot
S_{i_{k}}$ and consider the $C^*$-algebra $\mathcal{F}_A$, generated by elements of type $S_{\mu}P_iS^*_{\nu}$, where $|\mu|=|\nu|=k$,
$k=0,1,...,n,$ $i\in 1,...,n$. One can show that the Cuntz-Krieger algebra is in fact a crossed product $\mathcal{O}_A=C^*(\mathcal{F}_A,S)\cong
\mathcal{F}_A\times_{\delta}\mathbb{Z},$ where $S$ is an isometry introduced in \cite{Anton}. Taking $S_1=\psi_1,$
$S_2=\psi_2$ and a $2\times 2$-matrix $A$ with all entries equal to one, we get the Cuntz algebra $\mathcal{O}_2$ as a special case of the Cuntz-Krieger algebra. That means, $\mathcal{O}_2=C^*(\mathcal{F}_A,S)\cong
\mathcal{F}_A\times_{\delta}\mathbb{Z}$, where $\mathcal{F}_A$
is generated by monomials 
$\psi_{i_{1}}\cdot\cdot\cdot\psi_{i_{k}}\psi_i\psi_i^*\psi^*_{k_{k}}\cdot\cdot\cdot\psi^*_{j_{1}},$
$i=1,2$ and $S=\frac{1}{\sqrt{2}}(\psi_1+\psi_2).$ One can see that $\mathcal{F}_A\subset\mathcal{O}^{U(1)}_2$, but
$\mathcal{F}_A\neq \mathcal{O}^{U(1)}_2$ and $\mathcal{F}_A$ is not the image of embedding of the CAR algebra in $\mathcal{O}_2$, since it does not contain generators such as $\psi_1\psi^*_2$. 

Thus, the example described in this work shows that the Cuntz algebra can be represented as a crossed product of the CAR algebra by endomorphism induced by $\psi_1$, i.e.
$\mathcal{O}_2\cong\mathcal{A}_C\times_{\delta}\mathbb{Z}$, where $\mathcal{A}_C$ is a $U(1)$-invariant subalgebra in $\mathcal{O}_2$.


\end{document}